\documentclass[12pt]{amsart}
\usepackage{amsmath,amssymb,amsthm,hyperref,color,euscript,enumerate}

\newcommand{\1}{{\bf 1}}

\newcommand{\maru}[1]{{\ooalign{\hfil#1\/\hfil\crcr
\raise.167ex\hbox{\mathhexbox20D}}}}

\newcommand{\ruby}[2]{%
 \leavevmode
 \setbox0=\hbox{#1}%
 \setbox1=\hbox{\tiny #2}%
 \ifdim\wd0>\wd1 \dimen0=\wd0 \end{lemma}se \dimen0=\wd1 \fi
 \hbox{%
   \kanjiskip=0pt plus 2fil
   \xkanjiskip=0pt plus 2fil
   \vbox{%
     \hbox to \dimen0{%
       \tiny \hfil#2\hfil}%
     \nointerlineskip
     \hbox to \dimen0{\mathstrut\hfil#1\hfil}}}}

\newcommand{\la}{\langle}
\newcommand{\ra}{\rangle}

\newcommand{\Z}{\mathbb{Z}}
\newcommand{\C}{\mathbb{C}}

\newcommand{\Aut}{\mathrm{Aut}}

\newcommand{\be}{\beta}
\newcommand{\al}{\alpha}

\makeatletter \@addtoreset{equation}{section}

\theoremstyle{plain}
\newtheorem{theorem}{Theorem}[section]
\newtheorem{proposition}[theorem]{Proposition}
\newtheorem{lemma}[theorem]{Lemma}
\newtheorem{corollary}[theorem]{Corollary}

\theoremstyle{definition}
\newtheorem{definition}[theorem]{Definition}

\theoremstyle{remark}

\numberwithin{equation}{section}

\title[Classification of Ising vectors in $V_L^+$]{Classification of Ising vectors\\ in the vertex operator algebra $V_L^+$}
\author[H. Shimakura]{Hiroki Shimakura}
\address[H. Shimakura]{Department of Mathematics,
Aichi University of Education,
1 Hirosawa, Igaya-cho, Kariya-city, Aichi, 448-8542 Japan}%
\email {shima@auecc.aichi-edu.ac.jp}%
\date{}
\thanks{H.\ Shimakura was partially supported by Grants-in-Aid for Scientific Research (No. 23540013), JSPS}

\newcommand{\sfr}[2]{\leavevmode\kern-.1em
  \raise.5ex\hbox{\the\scriptfont0 #1}\kern-.1em
  /\kern-.15em\lower.25ex\hbox{\the\scriptfont0 #2}}

\begin{document}

\begin{abstract} Let $L$ be an even lattice without roots.
In this article, we classify all Ising vectors in the vertex operator algebra $V_L^+$ associated with $L$.
\end{abstract}
\maketitle


\section*{Introduction}
In vertex operator algebra (VOA) theory, the simple Virasoro VOA $L(1/2,0)$ of central charge $1/2$ plays important roles.
In fact, for each embedding, an automorphism, called a $\tau$-involution, is defined using the representation theory of $L(1/2,0)$ (\cite{Mi}).
This is useful for the study of the automorphism group of a VOA.
For example, this construction gives a one-to-one correspondence between the set of subVOAs of the moonshine VOA isomorphic to $L(1/2,0)$ and that of elements in certain conjugacy class of the Monster (\cite{Mi,Ho}).

Many properties of $\tau$-involutions are studied using Ising vectors, weight $2$ elements generating $L(1/2,0)$.
For example, the $6$-transposition property of $\tau$-involutions was proved in \cite{Sa} by classifying the subalgebra generated by two Ising vectors.
Hence it is natural to classify Ising vectors in a VOA.
For example, this was done in \cite{Lam,LSY} for code VOAs. 
However, in general, it is hard to even find an Ising vector.

Let $L$ be an even lattice and $V_L$ the lattice VOA associated with $L$.
Then the subspace $V_L^+$ fixed by a lift of the $-1$-isometry of $L$ is a subVOA of $V_L$.
There are two constructions of Ising vectors in $V_L^+$ related to sublattices of $L$ isomorphic to $\sqrt2A_1$ (\cite{DMZ}) and $\sqrt2E_8$ (\cite{DLMN,Gr}).

The main theorem of this article is the following:
\setcounter{section}{2}
\setcounter{theorem}{4}
\begin{theorem}\label{Conj} Let $L$ be an even lattice without roots and $e$ an Ising vector in $V_L^+$.
Then there is a sublattice $U$ of $L$ isomorphic to $\sqrt2A_1$ or $\sqrt2E_8$ such that $e\in V_{U}^+$.
\end{theorem}
\setcounter{section}{0}
We note that this theorem was conjectured in \cite{LSY} and that if $L/\sqrt2$ is even and if $L$ is the Leech lattice, then this theorem was proved in \cite{LSY} and in \cite{LS}, respectively.
We also note that if $L$ has roots then the automorphism group of $V_L^+$ is infinite, and $V_L^+$ may have infinitely many Ising vectors.

\medskip

In this article, we prove Theorem \ref{Conj}, and hence we classify all Ising vectors in $V_L^+$.
Our result shows that the study of $\tau$-involutions of $V_L^+$ is essentially equivalent to that of sublattices of $L$ isomorphic to $\sqrt2E_8$ (cf.\ \cite{GL,GL2}).

The key is to describe the action of the $\tau$-involution on the Griess algebra $B$ of $V_L^+$.
Let $e$ be an Ising vector in $V_L^+$ and $L(4;e)$ the norm $4$ vectors in $L$ which appear in the description of $e$ with respect to the standard basis of $(V_L^+)_2$ (see Section 2 for the definition of $L(4;e)$).
By \cite{LS}, the $\tau$-involution $\tau_e$ associated to $e$ is a lift of an automorphism $g$ of $L$.
We show in Lemma \ref{L4e} that $g$ is trivial on $\{\{\pm v\}\mid v\in L(4;e)\}$.
This lemma follows from the decomposition of $B$ with respect to the adjoint action of $e$ (\cite{LHY}), the action of $\tau_e$ on it (\cite{Mi}) and the explicit calculations on the Griess algebra (\cite{FLM}).
By this lemma, we can obtain a VOA $V$ containing $e$ on which $\tau_e$ acts trivially.
By \cite{LSY} $e$ is fixed by the group $A$ generated by $\tau$-involutions associated to elements in $L(4;e)$.
Hence $e$ belongs to the subVOA $V^A$ of $V$ fixed by $A$.
Using the explicit action of $A$, we can find a lattice $N$ satisfying $e\in V_N^+$ and $N/\sqrt2$ is even.
This case was done in \cite{LSY}.

\section{Preliminaries}

\subsection{VOAs associated with even lattices} \label{sec:3}
In this subsection, we review the VOAs $V_L$ and $V_L^+$ associated with even lattice $L$ of rank $n$ and their automorphisms.
Our notation for lattice VOAs here is standard (cf. \cite{FLM}).

Let $L$ be a (positive-definite) even lattice with inner product $\la\,\cdot\, , \,\cdot\,\ra$. 
Let ${H} = \C \otimes_{\Z} L$ be an abelian Lie algebra and $\hat{{H}} = {H} \otimes \C [t,t^{-1}] \oplus \C c$ its affine Lie algebra. 
Let ${\hat{{H}}}^-={H}\otimes t^{-1}\C[t^{-1}]$ and let $S(\hat{{H}}^-)$ be the symmetric algebra of $\hat{{H}}^-$. 
Then $M_H(1) =S(\hat{{H}}^-)\cong\C[h(m)\mid h \in {H}, m < 0]\cdot {\bf 1}$ is the unique irreducible $\hat{{H}}$-module such that $h(m) \cdot {\bf 1} = 0$ for $h \in {H}$, $m \ge 0$ and $c=1$, where $h(m) = h \otimes t^m$. 
Note that $M_H(1)$ has a VOA structure.

The twisted group algebra $\C\{L\}$ can be described as follows.
Let $\langle\kappa\rangle$ be a cyclic group of order $2$ and
$1\to \langle\kappa\rangle  \to \hat{L} \to L\to 1$
a central extension of $L$ by $\langle\kappa\rangle$ satisfying the commutator relation $[e^\alpha,e^\beta]=\kappa^{\langle \alpha, \beta \rangle}$ for $\alpha, \beta \in L$.
Let $L\to \hat{L}, \alpha\mapsto e^\alpha$ be a section and $\varepsilon(,):L\times L\to \langle\kappa\rangle$ the associated $2$-cocycle, that is, $e^\alpha e^\beta=\varepsilon(\alpha,\beta)e^{\alpha+\beta}$.
We may assume that $\varepsilon(\alpha,\alpha)=\kappa^{\langle\alpha,\alpha\rangle/2}$ and $\varepsilon(,)$ is bilinear by \cite[Proposition 5.3.1]{FLM}. 
The twisted group algebra is defined by
\[
\C\{L\}\cong \C[\hat{L}]/(\kappa+1) = {\rm Span}_{\C}\{e^\alpha\mid
\alpha\in L\},
\]
where $\C[\hat{L}]$ is the usual group algebra of the group $\hat{L}$.
The lattice VOA $V_L$ associated with $L$ is defined to be $M_H(1) \otimes \C\{L\}$ (\cite{Bo,FLM}). 

For any sublattice $E$ of $L$, let $\C\{E\}={\rm Span}_\C\{e^\alpha \mid \alpha\in E\}$ be a
subalgebra of $\C\{L\}$ and let $H_E=\C\otimes_\Z E$ be a subspace of $H= \C\otimes_\Z L$. Then the subspace
$S(\hat{H}_E^-)\otimes \C\{E\}$ forms a subVOA of $V_L$ and it is isomorphic to the lattice VOA $V_E$.

Let $O(\hat{L})$ be the subgroup of $\Aut(\hat{L})$ induced from $\Aut(L)$.
By \cite[Proposition 5.4.1]{FLM} there is an exact sequence of groups
\begin{equation*}
1\to{\rm Hom}(L,\Z/2\Z)\to O(\hat{L})\ \bar{\to}\ \Aut(L)\to 1.\label{EqC}
\end{equation*}
Note that for $f\in O(\hat{L})$
\begin{equation}
f(e^\alpha)\in\{\pm e^{\bar{f}(\alpha)}\}.\label{Eq:f}
\end{equation}
By \cite[Corollary 10.4.8]{FLM}, $f\in O(\hat{L})$ acts on $V_L$ as an automorphism as follows:
\begin{equation}
f(h_{i_1}(n_1)h_{i_2}(n_2)\dots h_{i_k}(n_k)\otimes e^\alpha)=\bar{f}(h_{i_1})(n_1)\bar{f}(h_{i_2})(n_2)\dots \bar{f}(h_{i_k})(n_k)\otimes f(e^\alpha),\label{Eq:fa}
\end{equation}
where $n_i\in\Z_{<0}$ and $\alpha\in L$.
Hence $O(\hat{L})$ is a subgroup of $\Aut(V_L)$.

Let $\theta$ be the automorphism of $\hat{L}$ defined by $\theta(e^\alpha)=e^{-\alpha}$, $\alpha\in L$.
Then $\bar{\theta}=-1\in\Aut(L)$.
Using (\ref{Eq:fa}) we view $\theta$ as an automorphism of $V_L$.
Let $V_L^+=\{v\in V_L\mid \theta (v)=v\}$ be the subspace of $V_L$ fixed by $\theta$. 
Then  $V_L^+$ is a subVOA of $V_L$.
Since $\theta$ is a central element of $O(\hat{L})$, the quotient group $O(\hat{L})/\langle\theta\rangle$ is a subgroup of $\Aut(V_L^+)$.
Note that $V_L^+$ is a simple VOA of CFT type.

Later, we will consider the subVOA of $V_L^+$ generated by the weight $2$ subspace.
\begin{lemma}\label{Lsub} {\rm (cf.\ \cite[Proposition 12.2.6]{FLM})} Let $L$ be an even lattice without roots.
Let $N$ be the sublattice of $L$ generated by $L(4)$.
Then the subVOA of $V_L^+$ generated by $(V_L^+)_2$ is $(V_N\otimes M_{H'}(1))^+$, where ${H}'=(\langle N\rangle_\C)^\perp$ in $\langle L\rangle_\C$.
\end{lemma}

\subsection{Ising vectors and $\tau$-involutions}
In this subsection, we review Ising vectors and corresponding $\tau$-involutions.

\begin{definition}
A weight $2$ element $e$ of a VOA is called an \textsl{Ising vector} if the vertex subalgebra generated by $e$ is isomorphic to the simple Virasoro VOA of central charge $1/2$ and $e$ is its conformal vector.
\end{definition}
For an Ising vector $e$, the automorphism $\tau_e$, called the \textsl{$\tau$-involution} or \textsl{Miyamoto involution}, was defined in (\cite[Theorem 4.2]{Mi}) based on the representation theory of the simple Virasoro VOA of central charge $1/2$ (\cite{DMZ}).

Let $V$ be a VOA of CFT type with $V_1=0$.
Then the first product $(a,b)\mapsto a\cdot b=a_{(1)}b$ provides a (nonassociative) commutative algebra structure on $V_2$.
This algebra $V_2$ is called the \textsl{Griess algebra} of $V$.
The action of $\tau_e$ on the Griess algebra was described in \cite{LHY} as follows:

\begin{lemma}\label{LHLY}{\rm \cite[Lemma 2.6]{LHY}} Let $V$ be a simple VOA of CFT type with $V_1=0$ and $e$ an Ising vector in $V$.
Then $B=V_2$ has the following decomposition with respect to the adjoint action of $e$:
\begin{equation*}
B=\C e\oplus B^e(0)\oplus B^e(1/2)\oplus B^e(1/16),\label{EqB}
\end{equation*} where $B^e(k)=\{v\in B\mid e\cdot v=kv\}$.
Moreover, the automorphism $\tau_e$ acts on $B$ as follows:
\begin{equation*}
1\quad {\rm on}\quad  \C e\oplus B^e(0)\oplus B^e(1/2),\qquad -1\quad {\rm on}\quad B^e(1/16).\label{taue}
\end{equation*}
\end{lemma}

In the proof of the main theorem, we need the following lemma:

\begin{lemma}\label{LLSY}{\rm \cite[Lemma 3.7]{LSY}} Let $V$ be a VOA of CFT type with
$V_1=0$. 
Suppose that $V$ has two Ising vectors $e$, $f$ and that $\tau_e=\mathrm{id}$ on $V$. 
Then $e$ is fixed by $\tau_f $, namely $e\in V^{\tau_f}$.
\end{lemma}

Let $L$ be an even lattice of rank $n$ without roots, that is, $L(2)=\{v\in L\mid \langle v,v\rangle=2\}=\emptyset$.
Then $(V_L^+)_1=0$, and we can consider the Griess algebra $B=(V_L^+)_2$ of $V_L^+$.
Let $\{h_i\mid 1\le i\le n\}$ be an orthonormal basis of $H=\C\otimes_\Z L=\langle L\rangle_\C$.
Set $L(4)=\{v\in L\mid \langle v,v\rangle=4\}$.
For $1\le i\le j\le n$ and $\alpha\in L(4)$, set $h_{ij}=h_i(-1)h_j(-1)\1$ and $x_\alpha=e^\alpha+e^{-\alpha}=e^\alpha+\theta(e^\alpha)$.
Note that $x_\alpha=x_{-\alpha}$.
\begin{lemma}\label{Lprod}{\rm \cite[Section 8.9]{FLM}} 
\begin{enumerate}[{\rm (1)}]
\item The set $$\{h_{ij}, x_\alpha\mid 1\le i\le j\le n,\ \{\pm\alpha\}\subset L(4)\}$$
is a basis of $B$.
\item The products of the basis of $B$ given in (1) are the following:
\begin{eqnarray*}
{h_{ij}}\cdot h_{kl}&=&\delta_{ik}h_{jl}+\delta_{il}h_{jk}+\delta_{jk}h_{il}+\delta_{jl}h_{ik},\\
{h_{ij}}\cdot x_\alpha&=&\langle h_i,\alpha\rangle\langle h_j,\alpha\rangle x_\alpha,\\
{x_\alpha}\cdot x_\beta&=&
\begin{cases}\varepsilon(\alpha,\beta)x_{\alpha\pm\beta} &{\rm if}\ \langle\alpha,\beta\rangle=\mp2,\\ \alpha(-1)^2\1& {\rm if}\ \alpha=\pm\beta,\\ 0 & {\rm otherwise}. \end{cases}
\end{eqnarray*}
\end{enumerate}
\end{lemma}

Let $\al\in L(4)$. 
Then the elements $\omega^+(\al)$ and $\omega^-(\al)$ of $V_L^+$ defined by
\begin{equation}
\omega^\pm(\al)=\frac{1}{16} \al(-1)^2\cdot \mathbf{1}\pm
\frac{1}4 x_\alpha\label{Def:Ising1}
\end{equation}
are Ising vectors (\cite[Theorem 6.3]{DMZ}).
The following lemma is easy:

\begin{lemma}\label{prp:3.2}
The automorphisms $\tau_{\omega^\pm(\al)}$ of $V_L^+$ act by 
\[
u\otimes x_\be\mapsto (-1)^ {\la \al, \beta \ra} u\otimes x_\be
\quad \text{ for } u\in M_H(1)\text { and }\be \in L.
\]
\end{lemma}

In general, the following holds:

\begin{proposition}{\rm \cite[Lemma 5.5]{LS}}\label{LS56} Let $L$ be an even lattice without roots and $e$ an Ising vector in $V_L^+$.
Then $\tau_e\in O(\hat{L})/\langle\theta\rangle$.
\end{proposition}

We note that the main theorem was proved if $L/\sqrt2$ is even as follows:

\begin{proposition}\label{TLSY}{\rm \cite[Theorem 4.6]{LSY}} Let $L$ be an even lattice and $e$ an Ising vector in $V_{L}^+$.
Assume that the lattice $L/\sqrt2$ is even.
Then there is a sublattice $U$ of $L$ isomorphic to $\sqrt2A_1$ or $\sqrt2E_8$ such that $e\in V_{U}^+$.
\end{proposition}

\section{Classification of Ising vectors in $V_{L}^+$}

Let $L$ be an even lattice of rank $n$ without roots and $e$ an Ising vector in $V_{L}^+$.
Then by Lemma \ref{Lprod} (1)
\begin{equation}
e=\sum_{i\le j} c^e_{ij}h_{ij}+\sum_{\{\pm\alpha\}\subset L(4)} d^e_{\{\pm\alpha\}}x_\alpha,\label{Eqe}
\end{equation}
where $c^e_{ij},d^e_{\{\pm\alpha\}}\in\C$.
Set $L(4;e)=\{\alpha\in L(4)\mid d^e_{\{\pm\alpha\}}\neq0\}$, ${H}_1=\langle L(4;e)\rangle_\C$ and $H_2=H_1^\perp$ in $H$.
Note that if $\alpha\in L(4;e)$ then $-\alpha\in L(4;e)$.
Without loss of generality, we may assume that $h_i\in {H}_1$ if $1\le i\le \dim H_1$.
Then $H_2={\rm Span}_\C\{h_j\mid \dim H_1+1\le j\le n\}$.

By Proposition \ref{LS56}, $\tau_e\in O(\hat{L})/\langle\theta\rangle$.
Since $e\in V_L$, we regard $\tau_e$ as an automorphism of $V_L$.
Then $\tau_e\in O(\hat{L})$, and set $g=\bar{\tau_e}\in\Aut(L)$.
Since $\tau_e$ is of order $1$ or $2$, so is $g$.
The following is the key lemma in this article:

\begin{lemma}\label{L4e} Let $\beta\in L(4;e)$.
Then $g(\beta)\in\{\pm\beta\}$.
\end{lemma}
\begin{proof} By (\ref{Eq:f}) and (\ref{Eq:fa}), 
\begin{equation}
\tau_e(x_\beta)\in \{\pm x_{g(\beta)}\}.\label{Eq:taue}
\end{equation}
On the other hand, $\tau_e(e)=e$, (\ref{Eq:fa}) and (\ref{Eqe}) show
\begin{equation}
\tau_e(d^e_{\{\pm\beta\}}x_\beta)=d^e_{\{\pm g(\beta)\}}x_{g(\beta)}.\label{Eq:taue2}
\end{equation}
By (\ref{Eq:taue}) and (\ref{Eq:taue2}),
\begin{equation}
\frac{d^e_{\{\pm g(\beta)\}}}{d^e_{\{\pm\beta\}}}\in\{\pm1\}.\label{Eq:taue3}
\end{equation}

Suppose $g(\beta)\notin\{\pm\beta\}$.
Then $x_\beta-\tau_e(x_\beta)$ is non-zero, and it is an eigenvector of $\tau_e$ with eigenvalue $-1$.
By Lemma \ref{LHLY}, we have \begin{equation}
e\cdot(x_\beta-\tau_e(x_\beta))=\frac{1}{16}(x_\beta-\tau_e(x_\beta)).\label{Eq:e_1}
\end{equation}

Let us calculate the image of both sides of (\ref{Eq:e_1}) under the canonical projection $\mu:(V_L^+)_2\to {\rm Span}_\C\{h_{ij}\mid 1\le i\le j\le n\}$ with respect to the basis given in Lemma \ref{Lprod} (1).
By (\ref{Eq:taue}) the image of the right hand side of (\ref{Eq:e_1}) under $\mu$ is $0$:
\begin{equation}
\mu \left(\frac{1}{16}(x_\beta-\tau_e(x_\beta))\right)=0.\label{Eq:e_2}
\end{equation}

Let us discuss the left hand side of (\ref{Eq:e_1}).
By Lemma \ref{Lprod} (2) and (\ref{Eq:taue3}), we have
\begin{align*}
e\cdot(x_\beta-\tau_e(x_\beta))&=\left(\sum_{i\le j} c^e_{ij}h_{ij}+\sum_{\{\pm\alpha\}\subset L(4)} d^e_{\{\pm\alpha\}}x_\alpha\right)\cdot\left(x_\beta-\tau_e(x_\beta)\right)\\ 
&\in d^e_{\{\pm\beta\}}\left(\beta(-1)^2\1-g(\beta)(-1)^2\1\right)+{\rm Span}_\C\{x_\gamma\mid \{\pm \gamma\}\subset L(4)\}.
\end{align*}
Thus 
\begin{align*}
\mu (e\cdot(x_\beta-\tau_e(x_\beta)))&=d^e_{\{\pm\beta\}}\left(\beta(-1)^2\1-g(\beta)(-1)^2\1\right)\\
&=d^e_{\{\pm\beta \}}\left(\beta-g(\beta)\right)(-1)(\beta+g(\beta))(-1)\1.
\end{align*}
This is not zero by $g(\beta)\notin \{\pm \beta\}$, which contradicts (\ref{Eq:e_1}) and (\ref{Eq:e_2}).
Therefore $g(\beta)\in\{\pm\beta\}$.
\end{proof}

For $\varepsilon\in\{\pm\}$, set $L(4;e,\varepsilon)=\{v\in L(4;e)\mid g(v)=\varepsilon v\}$, $L^{e,\varepsilon}=\langle L(4;e,\varepsilon)\rangle_\Z$, and ${H}_1^\varepsilon=\langle L^{e,\varepsilon}\rangle_\C$.
Since $g$ preserves the inner product, ${H}_1={H}_1^+\perp {H}_1^-$ and $g$ acts on ${H}_2={H}_1^\perp$.
Let ${H}_2^\pm$ be $\pm1$-eigenspaces of $g$ in ${H}_2$.
For $\varepsilon\in\{\pm\}$, let $W^\varepsilon$ be a lattice of full rank in $H_2^\varepsilon$ isomorphic to an orthogonal direct sum of copies of $2A_1$.
Then
\begin{equation}
M_{H_2^\varepsilon}(1)\subset V_{W^\varepsilon}.\label{Subset:W}
\end{equation}

\begin{lemma}\label{Lfix} The Ising vector $e$ belongs to the VOA $V_{L^{e,+}\oplus W^+}^+\otimes V_{L^{e,-}\oplus W^-}^+$, and $\tau_e={\rm id}$ on this VOA.
\end{lemma}
\begin{proof} By Lemma \ref{L4e}, $L(4;e)=L(4;e,+)\cup L(4;e,-)$.
Hence, by (\ref{Eqe}) and (\ref{Subset:W}),
\begin{equation}
e\in (V_{L^{e,+}}\otimes M_{{H}_2^+}(1)\otimes V_{L^{e,-}}\otimes M_{{H}_2^-}(1))^+\subset V_{L^{e,+}\oplus W^+\oplus L^{e,-}\oplus W^-}^+.\label{Eq:Ve}
\end{equation}
Since $g$ acts by $\pm1$ on $L^{e,\pm}\oplus W^\pm$, the subspace of (\ref{Eq:Ve}) fixed by $\tau_e$ is 
$$V_{L^{e,+}\oplus W^+}^+\otimes V_{L^{e,-}\oplus W^-}^+.$$
Since $e$ is fixed by $\tau_e$, we have the desired result.
\end{proof}

We now prove the main theorem.

\begin{theorem}\label{MT} Let $L$ be an even lattice without roots.
Let $e$ be an Ising vector in $V_L^+$.
Then there is a sublattice $U$ of $L$ isomorphic to $\sqrt2A_1$ or $\sqrt2E_8$ such that $e\in V_{U}^+$.
\end{theorem}
\begin{proof} Set $V=V_{L^{e,+}\oplus W^+}^+\otimes V_{ L^{e,-}\oplus W^-}^+$.
By Lemma \ref{Lfix}, $e$ belongs to $V$ and $\tau_e={\rm id}$ on $V$.
Let $A=\langle \tau_{\omega^\pm(\beta)}\mid \beta\in L(4;e)\rangle$.
By Lemma \ref{LLSY}, $e$ belongs to the subVOA $V^A$ of $V$ fixed by $A$.
Since $e$ is a weight $2$ element, it is contained in the subVOA generated by $(V^A)_2$.
By Lemmas \ref{Lsub} and \ref{prp:3.2} and (\ref{Subset:W}) (cf.\ (\ref{Eq:Ve})),
$$e\in V_{N^+\oplus K^+}^+\otimes V_{N^-\oplus K^-}^+\subset V_N^+,$$ where for $\varepsilon\in\{\pm\}$, $N^\varepsilon={\rm Span}_\Z\{v\in L(4;e,\varepsilon)\mid \langle v, L(4;e)\rangle\in 2\Z\}$, $K^\varepsilon$ is a lattice of full rank in $(\langle N^\varepsilon\rangle_\C)^\perp\cap(H_1^\varepsilon\oplus H_2^\varepsilon)$ isomorphic to an orthogonal direct sum of copies of $2A_1$, and $N=N^+\oplus K^+\oplus N^-\oplus K^-$.
Since $N$ is generated by norm $4$ and $8$ vectors, and the inner products of the generator belong to $2\Z$, the lattice $N/\sqrt2$ is even.
By Proposition \ref{TLSY}, there is a sublattice $U$ of ${N}$ isomorphic to $\sqrt2A_1$ or $\sqrt2E_8$ such that $e\in V_U^+$.
It follows from $K^+(4)=K^-(4)=\emptyset$ that ${N}(4)=N^+(4)\cup N^-(4)\subset L$.
Since $\sqrt2A_1$ and $\sqrt2E_8$ are spanned by norm $4$ vectors as lattices, we have $U\subset L$.
Hence $V_U^+$ is a subVOA of $V_L^+$.
\end{proof}

As an application of the main theorem, we count the total number of Ising vectors in $V_L^+$ for even lattice $L$ without roots.

Let us describe Ising vectors in $V_L^+$.
The Ising vector $\omega^\pm(\alpha)$ associated to $\alpha\in L(4)$ was described in (\ref{Def:Ising1}) as follows:
$$\omega^\pm(\al)=\frac{1}{16} \al(-1)^2\cdot \mathbf{1}\pm\frac{1}4 x_\alpha.$$

Let $E$ be an even lattice isomorphic to $\sqrt2E_8$ and $\{u_i\mid 1\le i\le 8\}$ an orthonormal basis of $\C\otimes_{\Z}E$.
We consider the trivial $2$-cocycle of $\C\{E\}$ for $V_E$.
Then for $\varphi\in {\rm Hom}(E,\Z/2\Z)(\cong (\Z/2\Z)^8)$
\begin{equation*}
\omega(E,\varphi)=\frac{1}{32}\sum_{i=1}^8u_i(-1)^2\cdot\1+ \frac{1}{32} \sum_{\{\pm\alpha\}\subset E(4)} (-1)^{\varphi(\alpha)}x_\alpha\label{Def:Ising2}
\end{equation*}
is an Ising vector in $V_E^+$ (\cite{DLMN,Gr}).
Since $E(4)$ spans $E$ as a lattice, $\omega(E,\varphi)=\omega(E,\varphi')$ if and only if $\varphi=\varphi'$.
Hence $V_E^+$ has $256$ Ising vectors of form $\omega(E,\varphi)$.
Thus $V_{\sqrt2A_1}^+$ and $V_{\sqrt2E_8}^+$ has exactly $2$ and $496$ Ising vectors, respectively (\cite[Proposition 4.2 and 4.3]{LSY}).

\begin{corollary} Let $L$ be an even lattice without roots.
Then the number of Ising vectors in $V_L^+$ is given by $$ |L(4)|+256\times |\{U\subset L\mid U\cong \sqrt2E_8\}|.$$
\end{corollary}
\begin{proof} Set $m=|L(4)|+256\times |\{E\subset L\mid E\cong \sqrt2E_8\}|$.
Theorem \ref{MT} shows that the number of Ising vectors in $V_L^+$ is less than or equal to $m$.
Let us show that there are exactly $m$ Ising vectors in $V_L^+$, that is, the Ising vectors $\omega^\pm(\alpha)$ and $\omega(E,\varphi)$ are distinct.
By Lemma \ref{Lprod} (1), $\omega^\varepsilon(\alpha)=\omega^{\delta}(\beta)$ if and only if $\alpha=\beta$ and $\varepsilon=\delta$.
Moreover, $\omega^\varepsilon(\alpha)\neq \omega(E,\varphi)$ for all $\alpha\in L(4)$, $L\supset E\cong \sqrt2E_8$ and $\varphi\in{\rm Hom}(E,\Z/2\Z)$.

Let $E_1,E_2$ be sublattices of $L$ such that $E_1\cong E_2\cong \sqrt2E_8$.
Let $\varphi_i\in {\rm Hom}(E_i,\Z/2\Z)$, $i=1,2$.
Then it follows from Lemma \ref{Lprod} (1) and $\langle E_i(4)\rangle_\Z=E_i$ that $\omega(E_1,\varphi_1)=\omega(E_2,\varphi_2)$ if and only if $E_1=E_2$ and $\varphi_1=\varphi_2$.
Therefore, there are exactly $m$ Ising vectors in $V_L^+$.
\end{proof}

\paragraph{\bf Acknowledgement.} The author thanks the referee for valuable advice.

\end{document}